\documentclass[a4paper,reqno,12pt]{amsart}

\usepackage[pages=all, color=black, position={current page.south}, placement=bottom, scale=1, opacity=1, vshift=5mm]{}

\usepackage[margin=1in]{geometry} 

\usepackage{amsmath}
\usepackage{amsthm}
\usepackage{amssymb}
\usepackage{tcolorbox}
\usepackage{commath}
\usepackage{amsfonts}
\usepackage{mathtools}
\usepackage{mathabx}
\usepackage{tabulary}
\usepackage{booktabs}
\usepackage[toc,page]{appendix}
\usepackage[mathscr]{euscript}
\usepackage{hyperref}
\usepackage{mathrsfs}
\newcommand{\R}{\mathbb{R}}
\newcommand{\Z}{\mathbb{Z}}
\usepackage[utf8]{inputenc}
\usepackage{hyperref}
\hypersetup{
	unicode,
	pdfauthor={Author One, Author Two, Author Three},
	pdftitle={A simple article template},
	pdfsubject={A simple article template},
	pdfkeywords={article, template, simple},
	pdfproducer={LaTeX},
	pdfcreator={pdflatex}
}

\usepackage[sort&compress,numbers,square]{natbib}

  \newcommand\vertarrowbox[3][6ex]{%
  \begin{array}[t]{@{}c@{}} #2 \\
  \left\uparrow\vcenter{\hrule height #1}\right.\kern-\nulldelimiterspace\\
  \makebox[0pt]{\scriptsize#3}
  \end{array}%
}
  \usepackage[shortlabels]{enumitem}
\let\olditem\item
\newlist{methods}{itemize}{1}
\setlist[methods]{%
    align=right,
    before=\changeitem,
    font=\bfseries,
    after=\let\item\olditem
}
\newcommand*{\changeitem}{%
    \renewcommand*{\item}[1][]{%
        \olditem[##1 :]
    }%
}
\theoremstyle{plain}
\newtheorem{theorem}{Theorem}
\newtheorem{corollary}[theorem]{Corollary}
\newtheorem{lemma}[theorem]{Lemma}

\newtheorem{proposition}[theorem]{Proposition}

\theoremstyle{definition}
\newtheorem{definition}[theorem]{Definition}

\newtheorem{remark}[theorem]{Remark}

\usepackage[mathscr]{euscript}
\usepackage{graphicx, color}
\graphicspath{{fig/}}
\newcommand{\ran}{\rangle}
\newcommand{\lan}{\langle}
\usepackage{algorithm, algpseudocode} 
\usepackage{mathrsfs} 

\newcommand\restr[2]{\ensuremath{\left.#1\right|_{#2}}}
\newcommand{\spa}{\;\;\;\;\;\;\;\;\;\;\;\;\;\;}

\newcommand{\til}{\Tilde}

\newcommand{\dl}{\partial}

\usepackage{bbm}
  \usepackage{nicematrix}
  \usepackage{pdfpages}

\usepackage{lscape}
\usepackage{subfig}
\usepackage{float}
\usepackage{plain}
\usepackage{lipsum}

\title[Control of Baouendi-Grushin eigenfunctions]{A Spectral Condition For The Control Of Eigenfunctions Of Baouendi-Grushin Type Operators}
\author[Mohammad Harakeh]
{Mohammad Harakeh}
\email{mohamadharakeh5@gmail.com}
\address{Institut Denis Poisson, UMR 7013 Universit\'e d'Orl\'eans-CNRS-Université de Tours,\\
Orl\'eans, France}

\author[L. Hillairet]
{Luc Hillairet}
\email{luc.hillairet@math.univ-orleans.fr}
\address{Institut Denis Poisson, UMR 7013 Universit\'e d'Orl\'eans-CNRS-Université de Tours,\\
Orl\'eans, France}
\thanks{This work benefited from the support of the PCR grant ADYCT- ANR-20-CE40-0017 from the Agence Nationale de la Recherche.}    

\date{\today}

\begin{document}
	\maketitle
	
	\begin{abstract}
          We prove a generic simplicity result on the multiplicity of the eigenvalues of the generalized Baouendi Grushin operator that
          implies the validity of concentration inequality for eigenfunctions.

	\end{abstract}

\section{Introduction}
Concentration inequalities for eigenfunctions are estimates of the probability that the eigenfunctions of a certain type of operator will be
concentrated in a particular region of the underlying space, more specifically whether the magnitude of the eigenfunctions on the manifold can
be controlled by its magnitude on some sub-domain.  Mathematically speaking, for a smooth connected manifold $M$, a possibly unbounded operator $T$
defined on $L^2(M)$ is said to satisfy the concentration inequality on a subset $N$ of $M$ (control region) if
  \begin{equation}\label{ci}
  \exists c(N)>0,\forall E\in\text{spec}(T),\forall \phi\in \text{ker}(T-E),~~\norm{\phi}_{{L^2}(M)}\leq c(N)\norm{\phi}_{{L^2}(N)}.
  \end{equation}

  This inequality is related to the observability question which is, usually a stronger result and has been studied in several papers. For instance,
  the authors in \cite{zbMATH06194652} prove the concentration inequality for eigenfunctions of the Schr\"odinger operator on a torus and 
  mention that the concentration inequality for eigenfunctions can be derived from the observability of the Schr\"odinger equation and Duhamel's formula.
  We refer to \cite{burq2004geometric,miller2012resolvent} for detailed studies about the link between observability and
  concentration inequalities for eigenfunctions.

  Observability and controllability for different types of operators on different domains have been widely studied
  (see for instance in  \cite{bourgain2012restriction,zbMATH05173335,burq2004geometric,dermenjian2022concentration,hassell2009eigenfunction}). More precisely in these references, the relation between the spectral problem and the underlying geometry is studied.
  Indeed, a well-known sufficient condition for concentration-type inequalities is the so-called Geometric Control Condition of Bardos-Lebeau-Rauch (GCC) that
  has been introduced in \cite{bardos1992sharp}. This condition says that all the trajectories of the generalized geodesic flow will enter the control
  region before some time. This condition is equivalent to observability in many cases such as the wave equation
  (see \cite{bardos1992sharp,burq2004smoothing}), but only sufficient in others such as for the Schr\"odinger equation (see \cite{lebeau1992controle}).
  For instance, the concentration inequality for eigenfunctions still holds when the only obstruction to GCC is the existence of flat cylinders of
  periodic orbits that do not intersect the control region (see \cite{Jaffard90,burq2004geometric,Marzuola2006,hassell2009eigenfunction}). 

  In all the previously stated works, the underlying operator is associated to a Riemannian Laplacian or more generally an elliptic
  pseudodifferential operator. In this work, we will study instead hypoelliptic operators in relation with sub-Riemannian geometry.
  Control and observability for this kind of operators have also been studied (see, for instance \cite{burq2022time,koenig2017non}) using
  involved techniques allowing for a fine understanding of the geometry of the phase-space. In \cite[Chapter~3]{letrouit2021equations}, the author
  considers the operator $P=-\dl_x^2-|x|^{2\gamma}\dl_y^2$ on $[-1,1]_x\times \mathbb T_y$, where $\mathbb T$ denotes the 1-dimensional torus and proves
  the following theorem  
\begin{theorem}
  [Letrouit \cite{letrouit2021equations}]\label{8564} Let $\gamma\geq 1$ and let $w=(-1,1)_x\times I$ for some interval $I$. Then, there
  exists $C,h_0>0$ such that for any $u\in D(\Delta_\gamma),$ and any $0<h\leq h_0$, \begin{equation}
        \norm{u}_{L^2(M)}\leq C\left(\norm{u}_{L^2(w)}+h^{-(\gamma+1)}\norm{(h^2\Delta_\gamma+1)u}_{L^2(M)}\right).
    \end{equation}
  \end{theorem}
  that implies the concentration inequality for eigenfunctions. The proof for this theorem is geometric and depends on some geometric control condition.

  In this paper,  we are interested in proving similar results by purely spectral methods that are reminiscent of the work with cylinders of periodic orbits.
  We identify a purely spectral condition amounting to a spectral simplicity condition,
  that implies the validity of concentration inequalities for eigenfunctions of Baouendi Grushin operators with control on an arbitrary
  horizontal strip of the infinite cylinder. In particular, denoting by $P_V$ the Bouendi-Grushin operator, we prove that
\begin{theorem}
    \label{mlml} If all the eigenvalues of $P_V$ are of multiplicity two, then the concentration inequality holds true.
\end{theorem}
This theorem will be restated as theorem \ref{conv}. The latter theorem is not restricted to Baouendi-Grushin operators, but rather
closely related to separation of variables.

We then study this multiplicity condition in a special class of operators. We recover the fact that the multiplicities of the eigenvalues of
the simple Grushin operator are not uniformly bounded and extend this result to our class (see section \ref{soae}). This
ensures that we cannot expect the spectral condition to be true in general and we thus study this condition generically in the class of
Bouendi-Grushin operators.

A generic result of simplicity of eigenvalues for elliptic operators was first introduced by Albert in his thesis \cite{albert1971topology}, see also
\cite{albert1975genericity}. Later, Uhlenbeck showed that the theorem does hold in all dimensions \cite{uhlenbeck1976generic}.
We prove here a variation on Albert's methods in \cite{albert1975genericity} that implies a similar result for the subelliptic generalized Baouendi Grushin operator.
In particular, we prove the following (see sect. \ref{secgencaspp} for the proper definition of the class of operators that is studied).
\begin{theorem}
    \label{brrr} The multiplicity of the eigenvalues of the generalized Baouendi-Grushin operator is generically $2$. 
\end{theorem}

The paper is organized as follows. In section \ref{dfgfd}, we introduce our framework and define the generalized Baouendi Grushin operator.
We then prove the spectral condition. In section \ref{soae}, we study the simple Grushin operator, and prove that the spectral condition is
not true in general. Finally in section \ref{secgencaspp}, we study the general case and prove the genericity result.

\section{Framework}\label{dfgfd}
 
\subsection{The Generalized Baouendi Grushin Operator}\label{subsecoffirsch}
Denote by $X$ the infinite cylinder in $\mathbb{R}^2$, $X=\mathbb{R}\times\mathbb{S}^1$, with fundamental domain $\mathbb R\times [-\pi,\pi]$, and by $w=\mathbb{R}\times[a,b]$, a horizontal strip along $X$. We use the standard Lebesgue measure $dxdy$ on $X$ and denote by $L^2(X)$ the corresponding Hilbert space. In $L^2(X)$ we denote by
$L^2_0(X)$ the orthogonal complement to the functions that depend only on $x$:
\begin{equation}
    \label{l^2_0}L^2_0(X)=\left\{u:X\rightarrow\mathbb R; \int_X |u(x,y)|^2dxdy<\mathcal{1}, \int_{\mathbb{S}^1}u(x,y)dy=0\right\}.
\end{equation}
The subset of $L^2_0(X)$ consisting of smooth functions with compact support will be denoted by $\mathcal{C}^\mathcal{1}_c(\mathbb R)$.

For $\gamma>0$, we introduce the set
$${\mathbb{V}_\gamma}:=\{V=|x|^{2\gamma}\til W;\til W\in\mathcal{C}^0_b(\mathbb R), \til W\geq 1\},$$
where $\mathcal{C}^0_b(\mathbb R)$ denotes the space of bounded continuous functions.

For $V\in {\mathbb{V}_\gamma}$, we set $X_1=\dl_x$ and $X_2=\sqrt{V}\dl_y$.  We define the generalized Baouendi-Grushin (hypoelliptic) operator as
\begin{equation}
    \label{gbgo}P_V:=-X_1^2-X_2^2=-\dl_x^2-V(x)\dl_y^2.
\end{equation}
The operator $P_V$ with domain $\mathcal{C}_{c,0}^\mathcal{1}(X)$ is essentially selfadjoint on $L^2_0(X)$ and the unique selfadjoint extension with domain
\begin{equation}
    \label{dogbgo}D=\{u\in L^2_0(X); \dl_x^2u\in L^2(X), V(x)\dl_y^2u\in L^2(X)\},
\end{equation}
has compact resolvent (see remark \ref{hl2bktba} below). Consequently its spectrum consists of an increasing sequence of positive eigenvalues with finite multiplicity,
that converges to $+\mathcal{1}$.
\subsection{One Dimensional Schr\"odinger Operator} 
For a non-negative continuous function $V$ satisfying $\lim_{|x|\rightarrow+\mathcal{1}}V(x)=+\mathcal{1}$ (in particular for $V\in\mathbb V_\gamma$), we define the one dimensional (parameter dependent) Schr\"odinger operator $$P_V^k:=-\dl_x^2+k^2V(x),$$ with domain $\mathcal{C}^\mathcal{1}_c(\mathbb R)$, defined on $\mathcal{H}=L^2(\R)$.
It is well known that for $k\neq 0$, the operator $(P_V^k,\mathcal{C}^\mathcal{1}_c(\mathbb R))$ is essentially selfadjoint, and that the
unique sselfadjoint extension with domain $\mathbb{D}_V$ satisfying $$\mathbb{D}_V\subset\{u\in H^1(\mathbb R), V^{\frac{1}{2}}u \in \mathcal{H}\},$$
has compact resolvent (see \cite{zbMATH01339094,zbMATH06136470} for instance). Consequently, the spectrum of $P^k_V$ consists of an
increasing sequence of positive eigenvalues with finite multiplicity, that converges to $+\mathcal{1}$, and $\mathcal{H}$ has a orthonormal basis that
consists of eigenfunctions. In the latter case, it is known that each eigenvalue is actually simple (see \cite{pankov2001introduction} for instance).

\section{Spectral Condition}
Let's first observe that there exists a basis of eigenfunctions of $P_V$ of the form
$$\phi(x,y)=\varphi(x)e^{iky}\spa k\in\mathbb Z^*,$$
with $\varphi\in \mathcal{H}$.
Indeed, the potential $V$ is independent of $y$, so $P_V$ and $\dl_y^2$ commutes and we can look for joint eigenfunctions of the two operators
(see \cite{ballentine2014quantum,comm}). The eigenfunctions of $\dl_y^2$ have the form $\varphi(x)e^{iky}$ with $k\in\mathbb Z$, so there exist $k\in\mathbb Z^*$
(the case $k=0$ is excluded since we work in $L^2_0(X)$) such that $\varphi_k(x)e^{iky}$ is an eigenfunction of $P_V$. Substituting in the eigenvalue equation,
this implies that $\varphi(x)$ is an eigenfunction of the one dimensional Schr\"odinger operator $P_V^k$.
Then $\varphi(x)=\varphi_{k,j}(x)$ corresponds to the $j^{th}$ eigenvalue of $P_V^k$ for some $j\in\mathbb N$.
For a fixed $k$, we can choose a family of orthonormal eigenfunctions $\{\varphi_{k,j}\}_{j\in\mathbb N}$ that form a Hilbert basis in $\mathcal{H}$.

It is straightforward that the functions $\{(x,y)\mapsto \varphi_{k,j}(x)e^{iky}\}_{k\in\mathbb Z^*,~j\in \mathbb N}$ is a Hilbert eigenbasis
for $L^2(\mathbb R\times \mathbb S^1)$, and so they cover all the eigenvalues of $P_V$. Applying this into the eigenvalue equation, we get that 
\begin{equation}
\label{uytuyt}\text{spec}(P_V)=\bigcup_{k\in\mathbb{Z}^*}\text{spec}(P_V^k).
\end{equation}
This method is called separation of variables and works for any operator $P_V$ such that the corresponding Schrödinger operators $P_V^k$ have compact resolvent.

We denote by $\varphi_{k,\ell}(x)$ an eigenfunction corresponding to the $\ell$-th eigenvalue of $P_V^k$. Since $P_V^k=P_V^{-k},~\varphi_{k,\ell}(x)e^{iky}$ and $\varphi_{k,\ell}(x)e^{-iky}$
are both eigenfunctions that correspond to the same eigenvalue of $P_V$ and so the multiplicity of any eigenvalue of $P_V$ is even (we recall that $k\neq 0$).

\begin{remark}
  \label{hl2bktba} Using this relation between $P_V$ and $P_V^k$, one can deduce that $P_V$ with domain $C_0^\mathcal{1}(X)$ is essentially selfadjoint and that
  the unique selfadjoint extension has compact resolvent.
\end{remark}

A sufficient condition for $P_V$ to satisfy the concentration inequality is given in the following proposition. 
\begin{theorem}\label{conv}
   Suppose that all the eigenvalues of $P_V$ are of multiplicity 2. Then the concentration inequality holds true for $w=\R\times ]a,b[:$
  \[
    \exists C>0,~\forall E \geq 0,~\forall \phi \in \ker(P_V-E),~~ \|\phi\|_{L^2(X)}\,\leq\, C \|\phi \|_{L^2(w)}.
  \]
  \end{theorem}
     \begin{proof}
       Let $E$ be an eigenvalue of $P_V$. If $\text{mult}(E)=2$, then any eigenfunction of $E$ can be written as
       $$\phi(x,y)=\alpha\varphi_{k,\ell}(x)e^{iky}+\beta\varphi_{k,\ell}(x)e^{-iky}=\varphi_{k,\ell}(x)(\alpha e^{iky}+\beta e^{-iky}),$$
       for some $\ell$ (see the paragraph before this proposition).  \\
       We will explicitly compute $\norm{\phi}_{L^2(w)}^2$. First, write
       $$\alpha=\alpha_0+i\alpha_1 \text{ and } \beta=\beta_0+i\beta_1,\spa\alpha_0,\alpha_1,\beta_0,\beta_1\in\mathbb R.$$
           Observe that $\phi(x,y)$ has the following expression:
           \begin{equation*}
    \begin{split}
      \phi(x,y)&=\varphi_{k,\ell}k(x)\left[(\alpha_0+i\alpha_1)(\cos(ky)+i\sin(ky))+(\beta_0+i\beta_1)(\cos(ky)-i\sin(ky))\right]\\
      &= \varphi_{k,\ell}(x)[(\alpha_0+\beta_0)\cos(ky)+(\beta_1-\alpha_1)\sin(ky)\\
      &+i\left((\alpha_1+\beta_1)\cos(ky)+(\alpha_0-\beta_0)\sin(ky)\right)].
    \end{split}
  \end{equation*}
  
Now, direct computation for the modulus of $\phi$ gives that \begin{equation*}
        |\phi(x)|^2=\Re(\phi(x))^2+\Im(\phi(x))^2=|\varphi_{k,\ell}(x)|^2\left[\kappa_1 \cos^2(ky)+\kappa_2\sin^2(ky)+2\kappa_3\cos(ky)\sin(ky)\right],
\end{equation*}
where $\kappa_1=(\alpha_0+\beta_0)^2+(\alpha_1+\beta_1)^2$, $\kappa_2=(\alpha_0-\beta_0)^2+(\alpha_1-\beta_1)^2$, and $$\kappa_3=(\alpha_0+\beta_0)(-\alpha_1+\beta_1)+(\alpha_0-\beta_0)(\alpha_1+\beta_1)=2(\alpha_0\beta_1-\alpha_1\beta_0).$$
We compute \begin{equation*}
    \begin{split}
        \norm{\phi}_{L^2(w)}^2&= \norm{\varphi_{k,\ell}}^2_{\mathcal{H}} \left[\kappa_1\int_a^b \cos^2(ky)dy+\kappa_2\int_a^b \sin^2(ky)dy+2\kappa_3\int_a^b \cos(ky)\sin(ky)dy\right]\\&=  \norm{\varphi_{k,\ell}}^2_{\mathcal{H}} \left[(\kappa_1-\kappa_2)\dfrac{f(k)}{2k}+(\kappa_1+\kappa_2)\left(\dfrac{b-a}{2}\right)+\kappa_3\dfrac{g(k)}{k} \right],
    \end{split}
\end{equation*}
with $f(k)=\sin(2bk)-\sin(2ak)$  and  $g(k)=\cos^2(ak)-\cos^2(bk).$ \\ Taking $a=-\pi$ and $b=\pi$, we deduce that
$$  \norm{\phi}_{L^2(X)}^2=\pi\norm{\varphi_{k,\ell}}^2_{\mathcal{H}} (\kappa_1+\kappa_2).$$
So, we get that, 
\begin{equation}
  \label{ghrinnn} \dfrac{\norm{\phi}_{L^2(w)}^2}{\norm{\phi}_{L^2(X)}^2}=
  \dfrac{\kappa_1-\kappa_2}{\kappa_1+\kappa_2}\dfrac{f(k)}{2\pi k}+\dfrac{b-a}{2\pi}+\dfrac{\kappa_3}{\kappa_1+\kappa_2}\dfrac{g(k)}{\pi k}.
\end{equation}
Note that the constants $\kappa_1,\kappa_2$ and $\kappa_3$ depend on $k$. The functions $f(k)$ and $g(k)$ are functions of sine and cosine so they are bounded.
Also, the term $|\frac{\kappa_1-\kappa_2}{\kappa_1+\kappa_2}|$ is bounded above by $1$ and so, the first term of (\ref{ghrinnn}) converges to $0$ as
$k\rightarrow\mathcal{1}$. \\
Now, observe that
$$\kappa_1+\kappa_2=2(\alpha_0^2+\beta_0^2+\alpha_1^2+\beta_1^2).$$
Then, we have
\begin{equation}
            \label{le3tarafla} \left|\dfrac{\kappa_3}{\kappa_1+\kappa_2}\right|\leq \dfrac{1}{2}
          \end{equation}
          Indeed,  we explicitly write
          $$\dfrac{\kappa_3}{\kappa_1+\kappa_2}=\dfrac{2(\alpha_0\beta_1-\alpha_1\beta_0)}{2(\alpha_0^2+\beta_0^2+\alpha_1^2+\beta_1^2)},$$
          and observe that, as 
$$(\alpha_0^2+\beta_0^2+\alpha_1^2+\beta_1^2)-2(\alpha_0\beta_1-\alpha_1\beta_0)=(\alpha_0-\beta_1)^2+(\alpha_1+\beta_0)^2\geq 0$$ 
we get that $$2(\alpha_0\beta_1-\alpha_1\beta_0)\leq \alpha_0^2+\beta_0^2+\alpha_1^2+\beta_1^2.$$ 
A similar computation holds for $\dfrac{\kappa_3}{\kappa_1+\kappa_2}$.
 
This implies (\ref{le3tarafla}).
      Thus, the last term of (\ref{ghrinnn}) converges to 0 as $k\rightarrow \mathcal{1}$. We deduce that
    \begin{equation}\label{kentlemm}
        \begin{split}
            \lim_{E\rightarrow\mathcal{1}}\left[\min_{0\neq\phi\in \text{ker}(P_V-E)}\left(\dfrac{\norm{\phi}_{L^2(w)}^2}{\norm{\phi}_{L^2(X)}^2}\right)\right]&=\lim_{k\rightarrow\mathcal{1}}\left[\min_{\alpha,\beta}\left(\dfrac{\kappa_1-\kappa_2}{\kappa_1+\kappa_2}\dfrac{f(k)}{2\pi k}+\dfrac{b-a}{2\pi}+\dfrac{\kappa_3}{\kappa_1+\kappa_2}\dfrac{g(k)}{\pi k}\right)\right]\\&=\dfrac{b-a}{2\pi}>0.
        \end{split}
    \end{equation}
It remains to prove that (\ref{kentlemm}) implies that the concentration inequality holds true. Suppose, to contrary, that the concentration inequality 
doesn't hold, that is, for all $c>0$, there exist $E\in\text{spec}(P_V)$ and $0\neq\phi\in \text{ker}(P_V-E)$ such that 
$c\norm{\phi}_{L^2(w)}^2<{\norm{\phi}_{L^2(X)}^2}$. Taking $c=c_n=n$, this implies that there exist a sequence of eigenvalues 
$(E_n)_{n\in\mathbb N}$ and a sequence of corresponding eigenfunctions $(\phi_n)_{n\in\mathbb N}$ of $P_V$ such that 
$$\min_{0\neq\phi\in \text{ker}(P_V-E_n)}\left(\dfrac{\norm{\phi}_{L^2(w)}^2}{\norm{\phi}_{L^2(X)}^2}\right)\leq 
\dfrac{\norm{\phi_n}_{L^2(w)}^2}{\norm{\phi_n}_{L^2(X)}^2}\leq \frac{1}{n}.$$ 

We then observe that 
\[
\forall n,~\min_{0\neq\phi\in \text{ker}(P_V-E_n)}\left(\dfrac{\norm{\phi}_{L^2(w)}^2}{\norm{\phi}_{L^2(X)}^2}\right) > 0.
\]
Indeed the previous computation tells us that the latter minimum coincides with 
\[
\min \left\{ \dfrac{\|u\|^2_{L^2(]a,b[)}}{\| u\|_{L^2(]-\pi,\pi[)}},~~u\,=\,\alpha e^{ik\cdot}\,+\,\beta e^{-ik\cdot}\right\}
\]
and this quantity is positive since a linear combination of $e^{ik\cdot}$ and $e^{-ik\cdot}$ cannot vanish on the interval $]a,b[$.

We deduce that,  necessarily, $E_n\rightarrow\mathcal{1}$ when $n\rightarrow\mathcal{1}$.
We get that 
$$\lim_{E_n\rightarrow\mathcal{1}}\left[\min_{0\neq\phi\in \text{ker}(P_V-E_n)}\left(\dfrac{\norm{\phi}_{L^2(w)}^2}{\norm{\phi}_{L^2(X)}^2}\right)\right]=0,$$
which contradicts (\ref{kentlemm}).
         \end{proof}

The preceding proposition depends on separating the variables of eigenfunctions of the operator and not really on the 
particular form of Baouendi Grushin operators so it works for any operator with similar properties.  

The condition in theorem \ref{conv} is not true in general. For a better vision of the problem lets study first a family of operators that 
contains the simple Grushin operator.

\section{Family $P_{x^2+s^2}$}\label{soae}
In this section, we study the family $P_{x^2+s^2}$ for $s\in\R$. 

Denote by $P_s=P_{x^2+s^2}$ and by $P_s^k=P_{x^2+s^2}^k$. We investigate the eigenvalues of $P_s$; we study the multiplicities of the eigenvalues according to $s$.
\begin{proposition}\label{specps}
The spectrum of $P_s$ is given by the set \begin{equation}
    \label{hymfrudtkld}\mathrm{spec}(P_s)=\{E^s_{k,n}=(2n+1)|k|+k^2s^2; n\in\mathbb N, k\in\mathbb Z^*\}.
\end{equation}  An orthonormal basis of eigenfunction corresponding to $E^s_{k,n}$ is given by $\phi_{k,n}(x,y)=\varphi_{k,n}(x)e^{iky},$ with $$\varphi_{k,n}(x)=|k|^{1/4}H^n\left(x\sqrt{|k|}\right) e^{\frac{-x^2|k|}{2}}e^{iky},$$ where $H^n$ is the Hermite polynomial of degree $n$.
\end{proposition}
\begin{proof}
Let  $\phi(x,y)=\varphi_k(x)e^{iky}$, $k\in\mathbb{Z}^*$, be an eigenfunction of $P_s$ and let $E^s$ be its corresponding eigenvalue. The eigenvalue equation $$P_s\phi=E^s\phi$$ implies that $E^s=E^s_k$ is an eigenvalue of $P_s^k$ with a corresponding eigenfunction $\varphi_k$.\\ Observe that, the eigenvalue equation implies that $E^s_k=E^0_k+k^2s^2$, where $E^0_k$ is an eigenvalue of the 1-D harmonic oscillator $-\dl_x^2+k^2x^2$, with corresponding eigenfunction $\varphi_k$. The spectrum of the harmonic oscillator is well known and given by $\{(2n+1)|k|; n\in\mathbb N,k\in\mathbb Z^*\}$. Moreover, the function $\varphi_{k,n}$ given by \begin{equation}
    \label{mpmp}\varphi_{k,n}(x)=|k|^{1/4}H^n(x\sqrt{|k|}) e^{\frac{-x^2|k|}{2}},
\end{equation} is an eigenfunction corresponding to $(2n+1)|k|$. Refer to \cite[Chapter~11]{hislop2012introduction} or \cite[Chapter~6]{zworski2022semiclassical} for details about the harmonic oscillator. \\Then, for every $k\in\mathbb{Z}^*$, the $n^{th}$ eigenvalue of $P_s^k$ is $E^s_{k,n}=(2n+1)|k|+k^2s^2$, with a corresponding eigenfunction given by (\ref{mpmp}).

Therefore, the spectrum of $P_s$ is given by (\ref{hymfrudtkld}), with a set of corresponding eigenvectors $\{\varphi_{k,n}(x)e^{iky}\}_{k\in \mathbb Z^*}$. Since these eigenfunctions span the space $L^2_0(X),$ they cover all the eigenvalues of $P_s$. We conclude.
\end{proof}
To study multiplicity of eigenvalues, it is usually helpful to study the Weyl law, which will be described here by studying the asymptotic behaviour of the counting function. The counting function $N_{P_s}$ takes a positive real number and counts the number of eigenvalues of $P_s$ less than or equal to this number. In other words, we can write for $E>0$, $$N_{P_s}(E)=\sum_{E^s_{k,n}\leq E}1,$$ where the sum is taken over the eigenvalues $E^s_{k,n}$ of $P_s$.
\begin{proposition}\label{wayllow}[Weyl Law] The following assertions hold true. 
    \begin{enumerate}
    \item  For $s=0$,  $N_{P_0}(E)\,=\, E\ln(E)\,+\,O(E)$ at infinity.
    \item For $s\neq0$, $ N_{P_s}(E)\,=\,E\ln(\sqrt{E})\,+\,O(E)$ at infinity.
\end{enumerate} 
\end{proposition}
\begin{proof} Denote by $\lceil .\rceil$ the ceiling function which takes a real number and gives the first integer greater than or equal to this number. Denote by  $\lfloor.\rfloor$ the  floor function which takes a real number and gives the first integer less than or equal to this number
 \begin{enumerate}
    \item 
We compute  
\begin{equation*}
  \begin{split}
   N_{P_0}(E)&=\text{card}\{(n,k)\in\mathbb{N}\times\mathbb{Z}^*/(2n+1)|k|\leq E\}\\&=2\text{card}\{(n,k)\in\mathbb{N}\times\mathbb{N}^*/(2n+1)|k|\leq E\}\\&=2\sum_{0<k\leq E}{\text{card}\left\{n\in\mathbb{N}/(2n+1)\leq\dfrac{E}{k}\right\}}\\&=2\sum_{0<k\leq E}{\left\lceil\dfrac{E}{2k}\right\rceil},
  \end{split}  
\end{equation*}
Now, since $$\dfrac{E}{2k}\leq\left\lceil\dfrac{E}{2k}\right\rceil\leq\dfrac{E}{2k}+1,$$ then $$E\sum_{0<k\leq E}{\dfrac{1}{k}}\leq2\sum_{0<k\leq E}{\left\lceil\dfrac{E}{2k}\right\rceil}\leq E\sum_{0<k\leq E}{\dfrac{1}{k}}+2E.$$
As $E\rightarrow\mathcal{1}$, $$\sum_{0< k\leq E}{\dfrac{1}{k}}\,=\,\ln(E)\,+\,O(1).$$
Since $E$ is negligible at infinity compared to $E\ln(E)$, we get
  $N_{P_0}(E)\,=\, E\ln(E)\,+\,O(E)$.
\item  We set $\alpha=min\left(E,\dfrac{\sqrt{E}}{s}\right),$ and we compute 
\begin{equation*}
  \begin{split}
   N_{P_s}(E)&=\text{card}\{(n,k)\in\mathbb{N}\times\mathbb{Z}^*/(2n+1)|k|+k^2 s^2\leq E\}\\&=2\text{card}\{(n,k)\in\mathbb{N}\times\mathbb{N}^*/(2n+1)|k|+k^2 s^2\leq E\}\\&=2\sum_{0<k\leq \alpha}{\text{card}\left\{n\in\mathbb{N}/(2n+1)\leq\dfrac{E-k^2 s^2}{k}\right\}}\\&=2\sum_{0<k\leq \alpha}{\left\lceil\dfrac{E-k^2 s^2}{2k}\right\rceil}. 
  \end{split}  
\end{equation*}
Since $$\dfrac{E-k^2 s^2}{2k}\leq\left\lceil\dfrac{E-k^2 s^2}{2k}\right\rceil\leq\dfrac{E-k^2 s^2}{2k}+1,$$ then 
\begin{equation*}
    \begin{split}
        E\left(\sum_{0<k\leq \alpha}{\dfrac{1}{k}}\right)-s^2\left(\sum_{0<k\leq \alpha}{k}\right)&\leq2\sum_{0<k\leq \alpha}{\left\lceil\dfrac{E-k^2 s^2}{2k}\right\rceil}\\&\leq E\left(\sum_{0<k\leq \alpha}{\dfrac{1}{k}}\right)-s^2\left(\sum_{0<k\leq \alpha}{k}\right)+2E.
    \end{split}
\end{equation*}
Finally, we get 
\begin{equation*}
    \begin{split}
        E\left(\sum_{0<k\leq \alpha}{\dfrac{1}{k}}\right)-\dfrac{s^2\lfloor\alpha\rfloor(\lfloor\alpha\rfloor+1)}{2}&\leq2\sum_{0<k\leq \alpha}{\left\lceil\dfrac{E-k^2 s^2}{2k}\right\rceil}\\&\leq E\left(\sum_{0<k\leq \alpha}{\dfrac{1}{k}}\right)-\dfrac{s^2\lfloor\alpha\rfloor(\lfloor\alpha\rfloor+1)}{2}+2E.
    \end{split}
\end{equation*}
As $E\rightarrow\mathcal{1},$ we have $\alpha\,=\,\frac{\sqrt{E}}{s},$ and so we get 
  $N_{P_s}(E)\,=\,E\ln(\sqrt{E})\,+\,O(E).$
\end{enumerate}
\end{proof}
\begin{corollary}\label{unifunbdd}
  If $s$ is fixed such that $s^2$ is rational, then the multiplicity of $P_{s}$ is not uniformly bounded. 
  \end{corollary}
    \begin{proof}
    If $s^2=0$, we write the prime factorisation $E^0_{k,n}=2^{k_0}p_1^{\alpha_1}...p_r^{\alpha_r}$ for an eigenvalue $E^0_{k,n}$. With the convention that $\sum_1^0=\sum_1^{-1}=0,$ we have 
    \begin{equation}
        \label{mults=0}\text{mult}(E)=2\left[\sum_{i=1}^r{\alpha_i}+\sum_{j=1}^{r-1}{\alpha_j\left(\sum_{k=j+1}^r{\alpha_k}\right)}+1\right].
    \end{equation}
   Indeed, for $s=0$, the eigenvalues are of the form $(2n+1)|k|$ with $n\in\mathbb N$ and $k\in\mathbb Z$. The factor $2$ outside the brackets is because of the fact that $E^0_{k,n}=E^0_{-k,n}$. Now, every prime number but $2$ (and that's why we distinguished $2$) is an odd integer, and the product of two odd integers is odd. So, $$\text{ the term $2n+1$ can be $\prod_{i=1}^r p_i^{j_i}$ for any } 0\leq j_i\leq\alpha_i.$$ The first sum on the right hand side of (\ref{mults=0}) represents the number of the cases with $j_i=0$ for all $i$ but one. The second term of the right hand side covers the number of the cases where $j_i$ is not zero at least for two $i's$, and $k$ is not $2^{k_0}$. Finally, the $1$ is for the case where $k=2^{k_0}$ ($j_i=\alpha_i$ for all $i$). This covers all the cases. Formula (\ref{mults=0}) implies that for $s=0$, the multiplicity is not bounded.\\ 

   Take now $s^2=\frac{p}{q}$, with $1$ as the greatest common factor of $p$ and $q$ and $p\neq 0$. Then, we have  $$\text{spec}(P_s)=\{(2n+1)|k|+k^2s^2,(n,k)\in\mathbb{N}\times\mathbb{Z}^{*}\}\subset\{\alpha+\beta s^2, \alpha,\beta\in\mathbb{Z}\}\subset\dfrac{1}{q}\mathbb{Z}.$$ Assume, for a contradiction, that the mutliplicity is bounded above by some $M$.
   Then, for any $E$, and since the spectrum is a subset of $\dfrac{1}{q}\mathbb{Z}$, we have
   \[
     N_{P_s}(2E)-N_{P_s}(E) \leq MqE.
   \]
   But the previous proposition implies that
   \[
     N_{P_s}(2E)-N_{P_s}(E)\,=\,E\ln \sqrt{E}\,+\,O(E).
   \]
   This yields the contradiction.
   \end{proof}
\begin{corollary}\label{poip}
    The multiplicity of the eigenvalues of the simple Grushin operator $P_{x^2}$ is not uniformly bounded.
\end{corollary}
\begin{proposition}\label{mult2}
    If $s^2$ is irrational, then the eigenvalues of $P_s$ are of multiplicity 2.
    \end{proposition}
    \begin{proof}
Suppose to contrary $P_s$ has an eigenvalue of multiplicity greater than $2$.
Then there exists $k,k'>0 , n,n'>0$ with $k^2\neq {k'}^2, n\neq {n'}$ such that $$(2n+1)|k|+k^2 s^2=(2n'+1)|k'|+k'^2 s^2.$$ Then, $$s^2=\dfrac{(2n'+1)|k'|-(2n+1)|k|}{k^2-k'^2},$$ which contradicts the fact that $s^2$ is irrational.
\end{proof}
Therefore, as a corollary of theorem \ref{conv}, we have
\begin{corollary}\label{propspec}
    If $s^2$ is irrational, then (\ref{ci}) holds.
\end{corollary}
To sum up, we showed that whenever $s^2$ is irrational the multiplicity of the eigenvalues of $P_s$ is 2, which implies by the spectral condition that the concentration inequality holds. Moreover, we proved that whenever $s^2$ is rational, the multiplicity is not uniformly bounded, and thus the spectral condition fails, in particular for $s=0$; the simple Grushin operator ($\gamma=1$ and $W=1$). This means that the spectral condition is not true in general. \\
This example inspires us to study the spectral condition generically for general Grushin operators.

\section{Baouendi Grushin Type}\label{secgencaspp}
In this section, we prove that the generic Baouendi-Grushin operator satisfies the multiplicity condition that we derived in section \ref{subsecoffirsch}.
The power $\gamma$ will be fixed throughout this section.

First, we introduce the convenient functional setting. We denote by $\mathcal{C}^0_b(\R)$ the set of continuous and bounded functions defined on $\R$ and define
$${\mathbb{V}_\gamma}=\{V=|x|^{2\gamma}\til W;\til W\in\mathcal{C}^0_b(\mathbb R); \til W\geq 1\}.$$
Observe that $\mathbb{V}_\gamma$ is a subset of the vector space $\{V=|x|^{2\gamma}\til W;\til W\in\mathcal{C}^0_b(\mathbb R)\}$ on which we define the following norm: for $V=|x|^{2\gamma}W\in{\mathbb{V}_\gamma}$, we set $\norm{V}_{{\mathbb{V}_\gamma}}:=\norm{W}_\mathcal{1}$.

We also define  $${\mathbb{W}_\gamma}=\{ W\in\mathcal{C}^\mathcal{1}_c(\mathbb R); \til W\geq 0\}$$
that will, loosely speaking, allow us to describe the local deformations of $\mathbb{V}_\gamma$.

Recall that $P_V$ denotes the generalized Baouendi-Grushin operator, and $P_V^k$ denotes the 1-dimensional parameter dependent schr\"odinger operator.

Let (P) be the property: $$\forall k,l\in\mathbb{Z}^*; k^2\neq l^2\Rightarrow\text{spec}(P_V^k)\cap\text{spec}(P_V^l)=\emptyset\spa \mathrm{\text{(P)}}.$$
We can see from (\ref{uytuyt}) that if (P) holds, then $\text{mult}(E)=2$ for all $E\in\text{spec}(P_V)$, thus the concentration inequality holds by
theorem \ref{conv}. In fact, the spectral condition we seek is a 'simplicity' result for the non-elliptic Baouendi Grushin operators that is not true
in complete generality (as shown in section \ref{soae}).  
Here we prove that for a generic $V$ in ${\mathbb{V}_\gamma}$, $\text{mult}(P_V)=2$. More precisely, if we denote by $\mathbb{V}_\gamma^g$ the set of the
\textit{good} $V's$ that satisfy the property (P);
then  $\mathbb{V}_\gamma^g$ is residual in $({\mathbb{V}_\gamma},\norm{.}_{{\mathbb{V}_\gamma}})$ (i.e.
dense and a countable intersection of open dense sets in ${\mathbb{V}_\gamma}$).

\subsection{Hellmann–Feynman Theorem}
We will construct a countable family of open and dense sets in ${\mathbb{V}_\gamma}$ such that $\mathbb{V}_\gamma^g$
is equal to the intersection of these sets. To prove the density, we will need at some point the Hellmann–Feynman theorem applied to the 
family $P_{V+t|x|^{2\gamma}W}^k$, where $W\in\mathbb W$ and $t>0$. An assumption in Hellmann–Feynman's theorem is for the eigenvalues and eigenfunctions 
of the perturbed operator to be analytic in $t$. Kato proved in \cite{kato2013perturbation} that the eigenquantities of operators that are holomorphic 
of type (A) (in particular $P_{V+t|x|^{2\gamma}W}^k$) are analytic (see \cite[Chapter~7]{kato2013perturbation}).  Thus, the Hellmann–Feynman theorem, 
that we know state and prove, is applicable.

\begin{lemma}[Hellmann–Feynman]\label{feyhel}
  Let $\lambda(t)$ be an eigenbranch of $P_{V+t|x|^{2\gamma}W}^k$, and denote by $u(t)$ a normalized eigenfunction branch of $\lambda(t)$. Then,
  $$\restr{\dfrac{d}{dt}}{t=0}\lambda(t) \,=\, k^2\int_{\R} |x|^{2\gamma}W(x)|u(0)(x)|^2\, dx.$$
\end{lemma}

\begin{proof}
    Write the eigenvalue equation: \begin{equation}
        \label{evefht}P_{V+t|x|^{2\gamma}W}^ku(t)=\lambda(t)u(t).
    \end{equation} Multiply by $u(t)$ and then differentiate in $t$ to get 
    \begin{equation}\label{fgfgfgfg}
        \begin{split}
            \frac{d}{dt}\lambda(t)&=\dfrac{d}{dt}\lan u(t),P_{V+t|x|^{2\gamma}W}^ku(t)\ran\\&= \left\lan \dfrac{d}{dt}u(t),P_{V+t|x|^{2\gamma}W}^ku(t)\right\ran+\left\lan u(t),\dfrac{d}{dt}\left(P_{V+t|x|^{2\gamma}W}^ku(t)\right)\right\ran\\&=\lambda(t)\dfrac{d}{dt}\lan u(t),u(t)\ran+\left\langle u(t),\left(\dfrac{d}{dt}P_{V+t|x|^{2\gamma}W}^k\right)u(t)\right\rangle\\&=\left\langle u(t),\left(\dfrac{d}{dt}P_{V+t|x|^{2\gamma}W}^k\right)u(t)\right\rangle.
        \end{split}
      \end{equation}
 Evaluating at $t=0$ gives the result.     
\end{proof}

\begin{remark}
    \label{rmkofhelfey} Following the same proof, we can show that whenever $u(t)$ and $v(t)$ are two orthogonal eigenfunction branches that correspond to the same 
eigenvalue branch $\lambda(t)$, we have 
$$\left\langle u(t),\left(\dfrac{d}{dt}P_{V+t|x|^{2\gamma}W}^k\right)v(t)\right\rangle=0.$$ 
This is because, by orthogonality, the left hand side of (\ref{fgfgfgfg}) is zero. This does not occur on the line since the eigenvalues are simple.
\end{remark}

\subsection{Generic Simplicity Result}
 Fix $k,l\in\Z^*$ such that $k^2\neq l^2$. We define the set 
$$\mathcal{O}_{k,l,n}^\gamma=\{V\in {\mathbb{V}_\gamma};\text{spec}_n(P^k_{V})\cap\text{spec}_n(P^l_{V})=\emptyset\}\subset {\mathbb{V}_\gamma},$$ 
where $\text{spec}_n(P^k_{V})$ (resp. $\text{spec}_n(P^k_{V})$) denotes the first $n$ eigenvalues of $P^k_{V}$ (resp. $P^k_{V}$).

  By definition, we have that for any $\gamma>0$, \begin{equation*}
      \label{ocao} \mathbb V_\gamma^g=\bigcap_{k,l,n}\mathcal{O}_{k,l,n}^\gamma.
  \end{equation*}
   We first prove that $\mathcal{O}_{k,l,n}^\gamma$ is open in $({\mathbb{V}_\gamma},\norm{.}_{{\mathbb{V}_\gamma}})$.
  \subsubsection{Openness Of $\mathcal{O}_{k,l,n}^\gamma$}
  To prove that $\mathcal{O}_{k,l,n}^\gamma$ is open in $({\mathbb{V}_\gamma},\norm{.}_{{\mathbb{V}_\gamma}})$ we need to prove the continuity of the eigenvalues of the operator $P_V^k$. For a fixed $k\in\Z^*$ and $V\in{\mathbb{V}_\gamma}$, denote by $\lambda_m^k(V)$ the $m^{th}$ eigenvalue of $P_V^k$.
  \begin{proposition}[Continuity Of Spectrum]\label{jbtamnalbert}
  Fix $\gamma>0,V\in{\mathbb{V}_\gamma}$ and $k\in\Z^*$. Let $(V_n)_{n\geq 1}$ be a sequence of functions in ${\mathbb{V}_\gamma}$ that converges to $V$ in $\norm{.}_{{\mathbb{V}_\gamma}}$. Then, for any $m$ and
  any $\epsilon>0$ there exist $n_{m,\epsilon}$, such that for all $n\geq n_{m,\epsilon}$,
  \begin{equation}
        \label{ktrmk2b}|\lambda_m^k(V_n)-\lambda_m^k(V)|<2\epsilon.
    \end{equation} 
    \end{proposition}
    \begin{proof}
      Write $V_n=V+|x|^{2\gamma}W_n$, where $W_n$ is a sequence of continuous bounded functions that converges uniformly to 0 on $\mathbb R$.
      Observe that
      \[
        \left|\frac{V_n-V}{V}\right |\,\leq\, |W_n|,
      \]
      since $V(x)/|x|^{2\gamma}\geq 1$.
      
For $F$ subset of the domain of $P_V^k$, denote by $\Lambda_V$ the following map 
$$\Lambda_V(F)=\underset{u\neq 0}{\underset{u\in F}{max}}\left\{\dfrac{\langle P_V^ku,u\rangle_\mathcal{H}}{\norm{u}_\mathcal{H}^2}\right\}.$$  
Fix $m$ and let $F$ be the subspace spanned by the first $m$ eigenvectors of $P_V$. We compute   
\begin{equation*}
          \begin{split}
            \langle P_{V_n}^ku,u\rangle_{\mathcal H}&=\langle P_{ V}^ku,u\rangle_{\mathcal H}+k^2\int_{\mathbb R}|V_n-V||u|^2\\
            &\leq \Lambda_V(F)+k^2\int_{\mathbb R}{V\frac{|V_n-V|}{V}|u|^2}\\
            &\leq (1+\|W_n\|_{\infty})\lambda_m^k(V).
          \end{split}
      \end{equation*}
      Taking the maximum over all functions $u\in F$, we get that
      $$\lambda_m^k(V_n)-\lambda_m^k(V)\,\leq\,\lambda_m^k(V)\|W_n\|_{\infty}.$$
      Exchanging the roles of $V_n$ and $V$, we obtain
      \[
       \lambda_m^k(V)-\lambda_m^k(V_n)\,\leq\,\lambda_m^k(V_n)\|W_n\|_{\infty}\,\leq\,\lambda_m^k(V)\|W\|_{\infty}(1+\|W_n\|_{\infty}). 
     \]
     Since $\|W_n\|_{\infty}$ tends to $0$, we conclude (\ref{ktrmk2b}).
    \end{proof}
We now prove that $\mathcal{O}_{k,l,n}^\gamma$ is open in $({\mathbb{V}_\gamma},\norm{.}_{{\mathbb{V}_\gamma}})$.
\begin{theorem}
    \label{openness} For any $\gamma>0,n\in\mathbb N$ and $k,l\in\Z^*$ fixed such that $k^2\neq l^2$. Then, the set $\mathcal{O}_{k,l,n}^\gamma$ is open in $({\mathbb{V}_\gamma},\norm{.}_{{\mathbb{V}_\gamma}})$.
    \end{theorem}
    \begin{proof}
   Take $V\in \mathcal{O}_{k,l,n}^\gamma$.
Let $(V_j)_{j\geq1}$ be a sequence in ${\mathbb{V}_\gamma}$ that converges to $V$ in $\norm{.}_{{\mathbb{V}_\gamma}}$. We first prove that there exists $j_0>0$ such that for all $j\geq j_0$, $V_j\in \mathcal{O}_{k,l,n}^\gamma$.  \\
   Recall that $\lambda_m^k({V})$ and $\lambda_m^l({V})$ denote the $m^{th}$ eigenvalue of $P_{V}^k$ and $P_{V}^l$ respectively. Since $(V_j)_{j\geq1}$ converges to $V$ in ${\mathbb{V}_\gamma}$, then by proposition \ref{jbtamnalbert}, for any $\epsilon>0$, there exists $j_k>0$ and $j_l>0$ such that for all $j\geq j_{k,l}:=\max\{j_k,j_l\}$ the following inequalities hold true \begin{equation}
       \label{ghribbb} |\lambda_m^k(V_j)-\lambda_m^k(V)|<2\epsilon, \text{ and } |\lambda_m^l(V_j)-\lambda_m^l(V)|<2\epsilon.
   \end{equation}
    Now, let $I=\{1,...,n\}$ and denote by $$\delta=\min_{i_1,i_2\in I}\left\{|\lambda_{i_1}^k(V)-\lambda_{i_2}^l(V)|\right\}>0.$$ For $\epsilon=\frac{\delta}{4}$, there exists $j_{k,l}$ such that for all $j\geq j_{k,l}$, (\ref{ghribbb}) implies that  
\begin{equation*}
          \begin{split}
              \delta&\leq |\lambda_{i_1}^k({V})-\lambda_{i_2}^l({V})|\\&\leq|\lambda_{i_1}^k({V})-\lambda_{i_1}^k({V_j})|+|\lambda_{i_1}^k(V_j)-\lambda_{i_2}^l(V_j)|+|\lambda_{i_2}^l({V_j})-\lambda_{i_2}^l({V})|\\&< \dfrac{\delta}{2}+|\lambda_{i_1}^k(V_j)-\lambda_{i_2}^l(V_j)|+\dfrac{\delta}{2}=\delta+|\lambda_{i_1}^k(V_j)-\lambda_{i_2}^l(V_j)|.
          \end{split}
      \end{equation*} This implies that $|\lambda_{i_1}^k(V_j)-\lambda_{i_2}^l(V_j)|>0$ i.e. $\lambda_{i_1}^k(V_j)\neq\lambda_{i_2}^l(V_j)$ for all $i_1,i_2\in I$, which means that 
      $$\text{spec}_n(P^k_{V_j})\cap\text{spec}_n(P^l_{V_j})=\emptyset.$$
      Therefore, for any $j\geq j_{k,l}$, $V_j\in \mathcal{O}_{k,l,n}^\gamma$. \\
      
The preceding statement implies that the complement of $\mathcal{O}_{k,l,n}^\gamma$ is closed and therefore, $\mathcal{O}_{k,l,n}^\gamma$ is open.
    \end{proof}

Now, we prove the density of $\mathcal{O}_{k,l,n}^\gamma$ in $({\mathbb{V}_\gamma},\norm{.}_{{\mathbb{V}_\gamma}})$.

\subsubsection{Density Of $\mathcal{O}_{k,l,n}^\gamma$}
We first give a quantitative version of the continuity of the spectrum, that will be useful in what follows.
For  $k\in\Z^*$ fixed, denote by $\kappa_m$ the distance from $\lambda_m^k(V)$ to the rest of the spectrum of $P_V^k$, i.e. 
$$\kappa_m=\text{dist}\left(\lambda_m^k(V),\text{spec}(P_V^k)\setminus \lambda_m^k(V)\right).$$ 

For any $W\in\mathbb W$ satisfying 
 \begin{equation}\label{condonw}
     \norm{|x|^{2\gamma}W}_\mathcal{1}<\dfrac{\kappa_m}{|k|^2},
 \end{equation}
 we define the intervals $J_+$ and $J_-$ as $$J_+=\left]\lambda_m^k(V)+|k|\norm{|x|^{2\gamma}W}_\mathcal{1},\lambda_m^k(V)+\kappa_m\right[,$$
 $$J_+=\left]\lambda_m^k(V)-\kappa_m,\lambda_m^k(V)-|k|\norm{|x|^{2\gamma}W}_\mathcal{1}\right[.$$ 

\begin{proposition}\label{soc}
Fix $\gamma>0$ and $k\in\Z^*$. For any $W\in\mathbb W$ satisfying (\ref{condonw}), for any $J\subset J_{-}\cup J_{+}$ we have $\mu\in J $ implies $\mu\notin\mathrm{spec}(P_{V+|x|^{2\gamma}W}^k)$.
\end{proposition}
  \begin{proof}
     Let $W\in\mathbb W$ satisfying (\ref{condonw}). $P_V^k$ is an unbounded selfadjoint operator. Then, for $\mu\in J$ (which is subset of the resolvent set of $P_V^k$), $(P_V^k-\mu)^{-1}$ is bounded normal operator and the spectral radius coincides with the norm of the resolvent, that is, \begin{equation}\label{albashtino} 
     \norm{(P_V^k-\mu)^{-1}}_{\mathcal{L}(\mathcal{H})}=\sup\left\{|\mu|, \mu\in\text{spec}((P_V^k-\mu)^{-1})\right\}=\dfrac{1}{\text{dist}(\mu,\text{spec}(P^k_V))}, \end{equation} where $\mathcal{L}(\mathcal{H})$ denotes the space of linear bounded functions from $\mathcal{H}$ to $\mathcal{H}$. \\
 Let $u\in\mathcal{H}$, and let $v=(P_V^k-\mu)^{-1}u$. We compute 
 \begin{equation*}
     \begin{split}
       \norm{(P_{V+|x|^{2\gamma}W}^k-P_V^k)(P_V^k-\mu)^{-1}u}_\mathcal{H}^2&= \norm{(P_{V+|x|^{2\gamma}W}^k-P_V^k)v}_\mathcal{H}^2\\
       &=k^4\norm{|x|^{2\gamma}Wv}^2_\mathcal{H}\leq k^4\norm{|x|^{2\gamma}W}_\mathcal{1}^2\norm{v}_\mathcal{H}^2\\
       &\leq\dfrac{k^4}{\text{dist}(\mu,\text{spec}(P_V^k))^2}\norm{|x|^{2\gamma}W}_\mathcal{1}^2\norm{u}_\mathcal{H}^2.
     \end{split}
 \end{equation*} Now, for $\mu\in J_-\cup J_+$, we have that \begin{equation}
     \label{mb3rfshmnwen}\text{dist}(\mu,\text{spec}(P_V^k))>|k|^2\norm{|x|^{2\gamma}W}_\mathcal{1}.
 \end{equation} 
 Therefore, $$\norm{(P_{V+|x|^{2\gamma}W}^k-P_V^k)(P_V^k-\mu)^{-1}}_\mathcal{L(\mathcal{H})}\leq\dfrac{k^2}{\text{dist}(\mu,\text{spec}(P_V^k))}\norm{|x|^{2\gamma}W}_\mathcal{1}<1.$$
 Then by Neumann lemma, $(I+(P_{V+|x|^{2\gamma}W}^k-P_V^k)(P_V^k-\mu)^{-1})$ is invertible. \\ 
 Now, we have
 \begin{equation*}
     \begin{split}
         I&=(P_V^k-\mu)(P_V^k-\mu)^{-1}\\&=(P_V^k-\mu)(P_V^k-\mu)^{-1}+(P_{V+|x|^{2\gamma}W}^k-\mu)(P_V^k-\mu)^{-1}-(P_{V+|x|^{2\gamma}W}^k-\mu)(P_V^k-\mu)^{-1}\\&=(P_{V+|x|^{2\gamma}W}^k-\mu)(P_V^k-\mu)^{-1}-(P_{V+|x|^{2\gamma}W}^k-\mu+\mu-P_V^k)(P_V^k-\mu)^{-1}\\&=(P_{V+|x|^{2\gamma}W}^k-\mu)(P_V^k-\mu)^{-1}-(P_{V+|x|^{2\gamma}W}^k-P_V^k)(P_V^k-\mu)^{-1}.
     \end{split}
 \end{equation*} This implies that $$I+(P_{V+|x|^{2\gamma}W}^k-P_V^k)(P_V^k-\mu)^{-1}=(P_{V+|x|^{2\gamma}W}^k-\mu)(P_V^k-\mu)^{-1}.$$
 We conclude that $P_{V+|x|^{2\gamma}W}^k-\mu$ is invertible and thus $\mu\notin\text{spec}(P_{V+|x|^{2\gamma}W}^k)$. 
  \end{proof}
In some other words, the preceding proposition says that under a small enough perturbation, the eigenbranches starting from $\lambda_m^k(V)$ remain in some interval containing no other eigenvalues of the unperturbed operator.

To prove that $\mathcal{O}_{k,l,n}^\gamma$ is dense, we have to show that the following assertion holds true: Let $V\in{\mathbb{V}_\gamma}$ such that $V\notin\mathcal{O}_{k,l,n}^\gamma$, then $$\forall\epsilon>0,\exists\til V\in\mathcal{O}_{k,l,n}^\gamma \text{ such that } \norm{V-\til V}_{{\mathbb{V}_\gamma}}<\epsilon.$$ Informally speaking, we prove that under a finite number of small enough perturbations of the operators $P_V^k$ and $P_V^l$, the perturbed operator will not have any eigenbranch in common (so it will be in $\mathcal{O}_{k,l,n}^\gamma$ and close enough to $V$). 
\begin{lemma}\label{milli}
 Let $g\in L^1_{loc}(\mathbb R)$ (locally integrable). If for all $W\in\mathbb W$, $$\int_\mathbb R W(x)g(x)dx=0,$$ then $g=0$ almost everywhere.
 \end{lemma}
 \begin{proof}
     Denote by $\mathbbm{1}_{[a,b]}$ the indicator function of the interval $[a,b]$, and by $\varphi$ the standard mollifier. For $\epsilon\leq 1$, the function $$W_\epsilon:=\varphi_\epsilon*\mathbbm{1}_{[a,b]}:=\frac{1}{\epsilon}\varphi\left(\frac{x}{\epsilon}\right)* \mathbbm{1}_{[a,b]}$$ is in $\mathbb W$: smooth, non-negative, its support is subset of $[a-\epsilon,b+\epsilon]\subset[a-1,b+1]$. Moreover, it converges as $\epsilon\rightarrow0$ pointwise to $\mathbbm{1}_{[a,b]}$. Now, since $g$ is locally integrable, we apply Lebesgue dominated convergence theorem to deduce that $$\lim_{\epsilon\rightarrow0}\int_\mathbb R W_\epsilon(x)g(x)dx=\int_a^b g(x)dx.$$ The left hand side of the preceding equation is $0$ by assumption, so $\int_a^b g(x)dx=0$. This is true for arbitrary $a,b$ which implies that $g=0$ a.e. in $\mathbb R$.
 \end{proof}
It is well known that the  Schr\"odinger operator on the line has simple eigenvalues (see \cite{pankov2001introduction}), but this is not the case on the circle (where it may have double eigenvalues), nor on $\mathbb R^n$ for $n>1$. So, for a moment, we forget that we are working on $\mathbb R$ and we prove the following lemma on $\mathbb R^n$ for $n>1$, which is a variation of Albert's arguments in \cite{albert1975genericity} (also, it is not hard to see that it remains true on $\mathbb S^1$). In the following lemma, $|x|^{2\gamma}$ will denote $\norm{x}^{2\gamma}$. \\ Remark that lemma \ref{milli} holds true on $\mathbb R^n$. 
\begin{lemma}
    \label{simpofeb} Fix $\gamma>0, V\in {\mathbb{V}_\gamma}$ and $k\in \mathbb Z^*$. Let $\lambda$ be an eigenvalue of $P_{V}^k$ of multiplicity $m$. Then, the following assertions hold true.
    \begin{enumerate}
        \item There exists $W\in\mathbb W$ such that $P_{V+t|x|^{2\gamma}W}^k$ has an eigenbranch starting from $\lambda$ of multiplicity strictly less than $m$.
        \item  If we denote by $\kappa=\mathrm{dist}(\lambda,\mathrm{spec}(P_V^k))$, then there exists $t_0>0$ and $W\in\mathbb W$ such that for all $0<t\leq t_0$,
           $P_{V+t|x|^{2\gamma}W}^k$ has $m$ simple eigenvalues in $I=]\lambda-\frac{\kappa}{2},\lambda+\frac{\kappa}{2}[$.
    \end{enumerate} 
    \end{lemma}
    \begin{proof}
    \begin{enumerate}
    \item Fix $W\in\mathbb W$. By analytic perturbation, the eigenvalue $\lambda$ splits into $m$ eigenbranches (not necessarily distinct) of $P^k_{V+t|x|^{2\gamma}W}$.\\
      Suppose that the eigenbranches are identical, and denote this eigenbranch by $\lambda(t)$ ($\lambda(0)=\lambda$). This means that we can find $m$ orthonormal eigenfunction branches $\{u_1(t),...,u_m(t)\}$ associated to $\lambda(t)$. If we denote by $E_\lambda$ the eigenspace of $\lambda$ in $P_V^k$, then $\mathcal{U}=\{u_1(0),...,u_m(0)\}$ is an orthonormal basis of $E_\lambda$. \\
      Denote by $\dot{q}$ the quadratic form, defined on $E_\lambda$ by $$\dot{q}(u)=k^2\int_{\mathbb R^n}|x|^{2\gamma}W(x)|u(x)|^2dx.$$ \\
      Hellmann–Feynman theorem (theorem \ref{feyhel}) implies that at $t=0$, we have $$\dot{\lambda}(0)=\dot{q}(u_i(0)), \spa \forall 1\leq i\leq m,$$ where the dot represents the derivative with respect to $t$. Moreover, using remark \ref{rmkofhelfey}, we get that for any $i\neq j$, 
      $$0=\dot{q}(u_i(0),u_j(0)),$$
      where we used the same notation for the corresponding symmetric bilinear form.
      Thus the matrix $A_\mathcal{U}:=[\dot{q}(u_i(0),u_j(0))]_{1\leq i,j\leq m}$ satisfies $A_\mathcal{U}=\dot{\lambda}(0)I,$ where $I$ is the $m\times m$ identity matrix. \\
      So, for any orthonormal basis $\mathcal{V}=\{v_1,...,v_m\}$ of $E_\lambda$, The matrix $A_\mathcal{V}$ is a multiple of the identity matrix. Indeed, the matrices $A_\mathcal{U}$ and $A_{\mathcal{V}}$ are related as follows: if we denote by $P$ the matrix of change of basis between $\mathcal{U}$ and $\mathcal{V}$, then, as $\mathcal{U}$ and $\mathcal{V}$ are sets of orthonormal vectors, $P$ is orthogonal (that is $P^tP=I$), and $$A_\mathcal{V}=P^tA_{\mathcal{U}}P=\dot{\lambda}(0)P^tP=\dot{\lambda}(0)I.$$ Now, fix $\mathcal{V}=\{v_1,...,v_m\}$ an orthonormal basis of $E_\lambda$, and suppose that for any $W\in\mathbb W$, the $m$ eigenbranches of $P^k_{V+t|x|^{2\gamma}W}$ are identical. Then, for any $W\in\mathbb W$, there exists a constant that depend on $W$, $c(W)$, such that \begin{equation}
          \label{ymklmlkio}k^2\int_{\mathbb{R}^n}|x|^{2\gamma}W(x)v_i(x)v_j(x)dx=c(W)\delta_i^j.
      \end{equation}
This implies, by lemma \ref{milli}, that for any $i\neq j$, 
\[
a.e\,x \in \R^n,~ |v_i(x)|^2\,=\,|v_j(x)|^2.
\]
Then, 

\begin{equation}
\label{aqazsqa} 
a.e\,x\in\R^n, v_1(x)=\pm v_2(x).
\end{equation}
Moreover, (\ref{ymklmlkio}) implies that almost everywhere in $\R^n$, 

\begin{equation}\label{qazsqse}
     v_1(x)v_2(x)=0.
\end{equation}
Thus, (\ref{aqazsqa}) and (\ref{qazsqse}) implies that for almost every $x\in\R^n$, $\pm v_1(x)^2=0$ which implies that $v_1=0$; 
a contradiction as $v_1$ is normalized.
 \item  We prove this by induction on $m$. 
Assume first that $m=2$. by part one of this lemma, there is ${W}\in\mathbb W$ such that $P_{V+t|x|^{2\gamma}{W}}^k$ has an eigenbranch of multiplicity strictly less than $2$. This means that the two eigenbranches are simple\footnote{For eigenbranches, simple means there are no two identical eigenbranches. They may intersect at a countable set of $t$ however.}. We choose $t$ small enough so that by proposition \ref{soc}, the eigenvalues of $P_{V+t|x|^{2\gamma}W}$ are simple and in $I$. 

        Suppose this is true for $m-1$. We prove it for $m$. Using part one of this lemma, there exists a $W_0\in\mathbb W$ such that $P_{V+t|x|^{2\gamma}W_0}^k$ has an eigenbranch of multiplicity strictly less than $m$. \\ 
Now, there might be several groups of identical eigenbranches. We enumerate them as 
$$\Lambda_1=\{\lambda_1^1(t)=...=\lambda_1^{m_1}(t)\},\Lambda_2=\{\lambda_2^1(t)=...=\lambda_2^{m_2}(t)\},... $$
        where $1\leq m_i\leq m-1$ for all $1\leq i\leq \frac{m}{2}$. Now, since two analytic functions are either identical for all $t$ or intersect only on countable set, and since the eigenbranch representing $\Lambda_1$ is different than those representing $\Lambda_2$, then by the analyticity of the eigenbranches, we can choose $t_0$ small enough so that by proposition \ref{soc}, has an eigenvalue of multiplicity $m_1$, an eigenvalue of multiplicity $m_2$, etc. in $I$. \\
By the induction hypothesis, there exists $W_1\in \mathbb W$, such that the eigenvalue of multiplicity $m_1$ of $P_{V+t_0|x|^{2\gamma}W_0}$ 
will now split into $m_1$ distinct eigenbranches of $P_{V+|x|^{2\gamma}(t_0W_0+tW_1)}$ in $I$. 
Now, as these eigenbranches were split by the first place from those in $\Lambda_i$ for $2\leq i\leq \frac{m}{2}$, then again by analyticity of these eigenbranches, 
they can not , under a perturbation, come back again identical (at least they are different at $t=0$).\\ 
Now, choosing $t_1$ small enough (so that the condition of proposition \ref{soc} is now true for $t_0W_0+t_1W_1$), $P_{V+|x|^{2\gamma}(t_0W_0+t_1W_1)}$ have $m_1$ simple eigenvalues and another eigenvalues of multiplicity possibly greater than 1 in $I$.\\ 
Proceeding with the same argument for any set of identical eigenbranches for the resulting perturbed operator, we conclude in the last step that we can choose $t_{i_0}$ small enough and $W_{i_0}\in\mathbb W$ such that $P^k_{V+|x|^{2\gamma}(t_0W_0+...+t_{i_0}W_{i_0})}$ has $m$ distinct eigenvalues in $I$ (where $i_0$ is definitely less than or equal $\frac{m}{2}$).

In fact, the preceding construction of the $W's$ implies the following statement: 
$\exists i_0\leq \frac{m}{2},\exists t_1,...,t_{i_0},\forall s_j\leq t_j,j\leq i_0$, $\exists W_j(s_j-1)$, such that 
\begin{equation}
            \label{constr1}P^k_{V+|x|^{2\gamma}(s_0W_0+s_1W_1(s_0)+...+s_{i_0}W_{i_0}(s_{i_0-1}))} \text{ has simple eigenvalues in }I.
        \end{equation}
      \end{enumerate}  This concludes the proof. 
  \end{proof}
Again, as we previously explained, the preceding theorem is important in the cases where the Schr\"odinger operator doesn't have simple eigenvalues, such as considering the case where $X$ is exchanged with the two dimensional flat torus (see subsection \ref{toruss}).
\begin{lemma}\label{kentlahala}
   Fix $\gamma>0, V\in{\mathbb{V}_\gamma}$, and $k,l\in\mathbb Z^*$ such that $k^2\neq l^2$. Let $\lambda$ be a common simple eigenvalue for $P_V^k$ and $P_V^l$ (simple in both spectrums).  Then, there exists $W\in\mathbb W$, such that the eigenbranch starting from $\lambda$ of $P^k_{V+tW}$ and the eigenbranch starting from $\lambda$ of $P^l_{V+tW}$ are not identical.
   \end{lemma}
    \begin{proof}
    Suppose that for any $W\in\mathbb W$, the eigenbranch of $P_{V+|x|^{2\gamma}tW}^k$ starting from $\lambda$ is identical to the eigenbranch of  $P_{V+|x|^{2\gamma}tW}^l$ starting from $\lambda$. Denote this eigenbranch by $\lambda(t)$ $(\lambda(0)=\lambda)$, and denote by $u(t,W)$ (resp. $v(t,W)$) a corresponding normalized eigenfunction branch for $P_{V+|x|^{2\gamma}tW}^k$ (resp. $P_{V+|x|^{2\gamma}tW}^l$). Then, if we denote by $E_{\lambda}^k$ and $E_{\lambda}^l$ the eigenspace that corresponds to $\lambda$ in $P_{V}^k$ and  $P_V^l$ respectively, then $u(0,W)$ and $v(0,W)$ are orthonormal basis for $E_{\lambda}^k$ and $E_{\lambda}^l$ respectively. \\ Applying  Hellmann–Feynman theorem at $t=0$, we get that \begin{equation}
    \label{akhyaakbe1} \dot{\lambda}(0)=k^2\int_\mathbb{R}|x|^{2\gamma}W(x)|u(0,W)(x)|^2dx=l^2\int_\mathbb{R}|x|^{2\gamma}W(x)|v(0,W)(x)|^2dx.
\end{equation}
Since $E_{\lambda}^k$ and $E_{\lambda}^l$ are one dimensional, then $u(0,W)$ and $v(0,W)$ are independent of $W$. This implies that for any $W\in \mathbb W$, \begin{equation}
    \label{akhyaakbe2} k^2\int_\mathbb{R}|x|^{2\gamma}W(x)|u(0)(x)|^2dx=l^2\int_\mathbb{R}|x|^{2\gamma}W(x)|v(0)(x)|^2dx.
\end{equation} By lemma \ref{milli}, (\ref{akhyaakbe2}) implies that $$k^2|x|^{2\gamma}|u(0)(x)|^2=l^2|x|^{2\gamma}|v(0)(x)|^2 \; \text{ a.e. }$$ Thus, as $u(0)$ and $v(0)$ are normalized, we get that $k=l$ which is a contradiction.
\end{proof}
We can prove now the density of $\mathcal{O}_{k,l,n}^\gamma$ in $({\mathbb{V}_\gamma},\norm{.}_{{\mathbb{V}_\gamma}})$.
\begin{theorem}
    \label{denseness} For any $\gamma>0,n\in\mathbb N$ and $k,l\in\Z^*$ fixed such that $k^2\neq l^2$, the set $\mathcal{O}_{k,l,n}^\gamma$ is dense in $({\mathbb{V}_\gamma},\norm{.}_{{\mathbb{V}_\gamma}})$.
\end{theorem}
\begin{proof}
  Let $V\in{\mathbb{V}_\gamma}$ with $V\notin \mathcal{O}_{k,l,n}^\gamma$. Require to prove that
  \begin{equation}
          \label{sdfdsdfds} \forall \epsilon>0, \exists \til V_\epsilon\in \mathcal{O}_{k,l,n}^\gamma  \text{ such that }   \norm{V-\til V_\epsilon}_{{\mathbb{V}_\gamma}}<\epsilon.
        \end{equation}  \\
        We know that $V\notin \mathcal{O}_{k,l,n}^\gamma$, so there is some common eigenvalues between $P_V^k$ and $P_V^l$ among the first $n$ eigenvalues of each,
        some of which maybe be with multiplicity greater than one. Take a common eigenvalue $\lambda$ of multiplicity $m_1$ and $m_2$ in $P_V^k$ and $P_V^l$ respectively. Recall that in this setting, $m_1=m_2=1$, but we give the proof for the general case (see paragraph before lemma \ref{simpofeb}).
        Let $$\kappa=\min\{\text{dist}(\lambda,\text{spec}(P_V^k)),\text{dist}(\lambda,\text{spec}(P_V^l))\},$$ and denote by
        $I=]\lambda-\frac{\kappa}{2},\lambda+\frac{\kappa}{2}[$.
        Using (\ref{constr1}), we obtain the following
        \begin{gather*}
          \exists i_0\leq \frac{m_1}{2} (\text{ in } \mathbb N^*),\exists t_0,...,t_{i_0};\forall s_j\leq t_j,0\leq j\leq i_0,\exists W_j=W_j(s_{j-1})\in\mathbb W\\
          \text{such that}~~P_{V+s_0W_0+...+s_{i_0}W_{i_0}}^k \text{ has $m_1$ simple eigenvalues in } I.
        \end{gather*}  

Now, again we do the same to obtain $m_2$ simple eigenvalues corresponding to $l$: 
\begin{gather*}
\exists i_1\leq \frac{m_1}{2}+\frac{m_2}{2} (\text{ in } \mathbb N^*),\exists t_{0},...,t_{i_1};\forall s_j\leq t_j,0\leq j\leq i_1,\exists W_j:=W_j(s_{j-1})\in\mathbb W;\\ 
\text{such that} ~~P_{V+s_0W_0+...+s_{i_1}W_{i_1}}^l \text{ has $m_2$ simple eigenvalues in } I.
\end{gather*}

By stability of simple eigenvalues under small perturbations, we get 
\begin{gather*}
\exists i_1\leq \frac{m_1}{2}+\frac{m_2}{2} (\text{ in } \mathbb N^*),\exists \til t_{0},...,\til t_{i_1};\forall s_j\leq \til t_j,0\leq j\leq i_1,
\exists W_j:=W_j(s_{j-1})\in\mathbb W;\\  
\text{such that}~~P_{V+s_0W_0+...+s_{i_1}W_{i_1}}^k \text{ and } P_{V+s_0W_0+...+s_{i_1}W_{i_1}}^l \text{ has $m_1$ and $m_2$ simple eigenvalues in } I \text{ resp. }.
\end{gather*}

Now, suppose among these simple eigenvalues there are $m$ common eigenvalues. Then, using now lemma \ref{kentlahala},we obtain 
\begin{gather*}
\exists i_2\leq \frac{m_1}{2}+\frac{m_2}{2}+m (\text{ in } \mathbb N^*),\exists \til t_{0},...,\til t_{i_2};\forall s_j\leq \til t_j,0\leq j\leq i_2,
\exists W_j:=W_j(s_{j-1})\in\mathbb W;\\
\text{such that}~~P_{V+s_0W_0+...+s_{i_2}W_{i_2}}^k \text{ and } P_{V+s_0W_0+...+s_{i_2}W_{i_2}}^l \text{ has no common eigenvalues in } I \text{ resp. }.
\end{gather*}

Now, this is true for any $\lambda$ in common between $P_V^k$ and $P_V^l$. Now, applying the same argument for the second common eigenvalue, then the third, etc..., 
it yields the following: 

\begin{gather*}
\exists \varsigma\in\mathbb N^*,\exists t_0,...,t_{\varsigma},\forall s_j\leq t_j,j\leq \varsigma, ~\exists W_j=W_j(s_j-1),\\
\text{such that}~~\left(\bigcap_{r=k,l}\text{spec}_n\left(P^r_{V+|x|^{2\gamma}(s_0W_0+s_1W_1(s_0)+...+s_{\varsigma}W_{\varsigma}(s_{\varsigma-1}))}\right)\right)=\emptyset.
\end{gather*}

Observe that, after dealing with the first intersection, we can freely deal with the next one; the eigenbranches that will be extracted from the simple 
(with respect to $k$ or $l$) disjoint (with respect to $k$ and $l$) eigenvalue we dealt with in the previous step won't come back identical because 
they were split at the first place (by analyticity). 

So, what we do to conclude, is that for any $j\leq \varsigma$, we choose $s_j$ small enough so that 
$$|s_j|\norm{|x|^{2\gamma}W_j(s_{j-1})}_{\mathcal{1}}<\dfrac{\epsilon}{\varsigma}.$$ T
his construction implies (\ref{sdfdsdfds}). 
\end{proof}

Finally, We prove the density of $\mathbb V_\gamma^g$ in $({\mathbb{V}_\gamma},\norm{.}_{{\mathbb{V}_\gamma}})$. 
\subsubsection{Density Of $\mathbb V_\gamma^g$ And Main Result} We define a Baire space and state Baire's category theorem.
\begin{definition}\label{Bairespace}
    A topological space is called a Baire space if every countable intersection of dense open sets is dense.
\end{definition}
\begin{lemma}[Baire's Category Theorem]\label{bct}
    Every completely metrizable topological space is a Baire space. 
    \end{lemma}
    \begin{proof}
        Refer to \cite{zbMATH03479763} for a proof.
    \end{proof}
\begin{corollary}\label{appobct}
    For $V=|x|^{2\gamma}W\in{\mathbb{V}_\gamma}$, we recall the norm $\norm{V}_{{\mathbb{V}_\gamma}}:=\norm{W}_\mathcal{1}$. For any $\gamma>0$, the space ${\mathbb{V}_\gamma}$ equipped with the norm $\norm ._{\mathbb{V}_\gamma}$ is a Baire space.
    \end{corollary}
    \begin{proof}
      We prove that the normed space $({\mathbb{V}_\gamma},\norm._{{\mathbb{V}_\gamma}})$ is complete. \\
      Let $(V_n)_{n\geq1}=(|x|^{2\gamma}W_n)_{n\geq1}$ be a Cauchy sequence in $({\mathbb{V}_\gamma},\norm._{{\mathbb{V}_\gamma}})$. 
Then, for any $\epsilon>0$, there exist $n_0\in\mathbb N$ such that for all $m,n\geq n_0$, we have 
$$\norm{V_n-V_m}_{{\mathbb{V}_\gamma}}<\epsilon.$$ 
This implies that for any $\epsilon>0$, there exist $n_0\in\mathbb N$ such that for all $m,n\geq n_0$, 
$$\norm{W_n-W_m}_\mathcal{1}=\norm{V_n-V_m}_{{\mathbb{V}_\gamma}}<\epsilon,$$ 
which gives that $(W_n)_{n\geq1}$ is a cauchy sequence in $(\mathcal{C}^0_b(\mathbb R),\norm._{\mathcal{1}})$ which is a complete space. 
So, $W_n$ converges to some $W$ in $\mathcal{C}^0_b(\mathbb R)$ uniformly. Finally, we get, for $V=|x|^{2\gamma}W$ 
(which is clearly in ${\mathbb{V}_\gamma}$), that $$\norm{V_n-V}_{{\mathbb{V}_\gamma}}=\norm{W_n-W}_\mathcal{1}\rightarrow0.$$
 We conclude that $(V_n)_{n\geq1}$ is convergent in $({\mathbb{V}_\gamma},\norm._{{\mathbb{V}_\gamma}})$ and therefore the space $({\mathbb{V}_\gamma},\norm._{{\mathbb{V}_\gamma}})$ is a complete metric space (metric induced from norm). Using lemma \ref{bct}, $({\mathbb{V}_\gamma},\norm._{{\mathbb{V}_\gamma}})$ is a Baire space.
    \end{proof}  
We deduce the density of $\mathbb{V}_\gamma^g$ in $({\mathbb{V}_\gamma},\norm._{{\mathbb{V}_\gamma}})$:
\begin{theorem}\label{mric1}
  For any $\gamma>0$ fixed, the set $\mathbb{V}_\gamma^g$ is dense in $({\mathbb{V}_\gamma},\norm ._{{\mathbb{V}_\gamma}})$.
\end{theorem}
 \begin{proof}
     By theorem \ref{openness} and theorem \ref{denseness}, $\mathcal{O}_{k,l,n}^\gamma$ is open and dense in ${\mathbb{V}_\gamma}$. Then, since $({\mathbb{V}_\gamma},\norm ._{{\mathbb{V}_\gamma}})$ is a Baire space by corollary \ref{appobct}, the intersection of $\mathcal{O}_{k,l,n}^\gamma$ which is equal to $\mathbb{V}_\gamma^g$ is dense in $({\mathbb{V}_\gamma},\norm ._{{\mathbb{V}_\gamma}})$.
 \end{proof}
 Recall that we say that a subset is residual in a metric space if it is the countable intersection of open dense sets, and it is dense. Our main result is a direct corollary of theorems \ref{openness}, \ref{denseness} and \ref{mric1}:
\begin{corollary}
    \label{finalll} For any $\gamma>0$ fixed, the set $\mathbb{V}_\gamma^g$ is residual in $({\mathbb{V}_\gamma},\norm ._{{\mathbb{V}_\gamma}})$.
\end{corollary}
\subsection{Validity Of Result On Torus}\label{toruss}
    The same results hold on a torus. We only introduce the set up, as all the arguments will be almost the same.

     Denote by $\mathbb T^2=\mathbb S^1\times \mathbb S^1$ the two dimensional flat torus with fundamental domain $[-\pi,\pi]_x\times[-\pi,\pi]_y$. For any function $f:\mathbb T^2\rightarrow\mathbb R$ (or $\mathbb C$), there corresponds one, and only one ($2\pi$-)periodic function $\til f:\mathbb R^2\rightarrow \mathbb R$ (or $\mathbb C$) given by $\til f(x,y)=f(e^{ix},e^{iy})$. 
     
     Recall that on the infinite cylinder, excluding the functions that does not depend on $y$ was necessary for $P_V$ to have discrete spectrum. Here, this is not an issue; the torus is a compact manifold and so the subelliptic operator will always have discrete spectrum. For this, we consider the space $$L^2(\mathbb T^2)=L^2(\mathbb T^2,\mathbb R)=\left\{f:\mathbb T^2\rightarrow \mathbb R;\norm{f}_{L^2(\mathbb T^2)}:=\norm{\til f}_{L^2([-\pi,\pi])}<\mathcal{1}\right\}.$$ 
     
     Now, observe that having a compact support for $W$ was important so that the coefficient $|x|^{2\gamma}$ wouldn't cause any problem when integrating over $\mathbb R$. Here however, as we work on $\mathbb S^1$, this is not longer an issue. For this, we denote by $\mathcal{C}^0(\mathbb S^1)$ the space of continuous $2\pi$-periodic functions on $\mathbb R$; $$\mathcal{C}^0(\mathbb S^1)=\{f:\mathbb S^1\rightarrow \mathbb R;\til f\in\mathcal{C}^0(\mathbb R)\}.$$
     We define the Baouendi Grushin operator on $\mathbb T^2$ for $V\in \mathcal{C}^0(\mathbb S^1)$ that looks like $|x|^{2\gamma}$ near $0$,
     and vanishes nowhere in $]-\pi,\pi[$ but on $0$. A typical example would be $V(x)=(4\sin^2\left(\frac{x}{2}\right))^{\frac{\gamma}{2}}$.
     For this, it is reasonable to define the set
     $$\til{\mathbb V}_\gamma=\left\{V(x)=\left(4\sin^2\left(\frac{x}{2}\right)\right)^{\frac{\gamma}{2}} W(x); W\in\mathcal{C}^0(\mathbb S^1),W\geq1\right\}.$$
     On $\til{\mathbb{V}_\gamma}$, we put the following norm: for $V(x)=\left(4\sin^2\left(\frac{x}{2}\right)\right)^{\frac{\gamma}{2}} W(x) \in\til{\mathbb V}_\gamma$,
     $\norm{V}_{\til{\mathbb V}_\gamma}=\norm{W}_\mathcal{1}$.

     For $V\in\til{\mathbb V}_\gamma,$ we denote by $\til P_V$ the operator $\til P_V=-\dl_x^2-V(x)\dl_y^2,$ defined on
     $$\til D=\{u\in L^2(\mathbb T^2);\dl_x^2u\in L^2(\mathbb T^2),V(x)\dl_y^2\in L^2(\mathbb T^2)\}.$$
     Take a strip of the torus, $$w=[-\pi,\pi]_x\times[a,b]_y\subset[-\pi,\pi]_x\times[-\pi,\pi]_y.$$

     Finally, we define, for $V\in\til{\mathbb V}_\gamma,$ the one dimensional Schr\"odinger operator as
     $\til P_V^k=-\dl_x^2-k^2V(x),$ with domain $$\til{\mathbb D}_V=\{u\in H^1(\mathbb S^1); V^{\frac{1}{2}}u\in L^2(\mathbb S^1)\}.$$

     In this setting we obtain the following theorem.
     \begin{theorem}
       For a generic $V\in \til{\mathbb{V}}_\gamma$, the spectrum of the corresponding Baouendi-Grushin operator on the torus has multiplicity at most $2$
       and, for any strip $w$, it satisfies the concentration inequality for eigenfunctions.
     \end{theorem}
     
\bibliographystyle{plain}
\bibliography{biblio.bib}
\end{document}